\documentclass[leqno,12pt]{article} 
\usepackage{amsmath, amssymb}
\usepackage{amsthm} 
\newcommand\supp{\mathop{\rm supp}}

\theoremstyle{plain} 
\newtheorem{theorem}{\indent\sc Theorem}[section]
\newtheorem{lemma}[theorem]{\indent\sc Lemma}
\newtheorem{corollary}[theorem]{\indent\sc Corollary}
\newtheorem{proposition}[theorem]{\indent\sc Proposition}

\theoremstyle{definition} 
\newtheorem{definition}[theorem]{\indent\sc Definition}
\newtheorem{remark}[theorem]{\indent\sc Remark}

\makeatletter
\def\address#1#2{\begingroup
\noindent\parbox[t]{7.8cm}{%
\small{\scshape\ignorespaces#1}\par\vskip1ex
\noindent\small{\itshape E-mail address}%
\/: #2\par\vskip4ex}\hfill%
\endgroup}%
\makeatother
\title{Fractional series operators on $\mathbb{Z}^n$} 
\author{
\textsc{Pablo Rocha} 
}
\date{} 
\begin{document}

\maketitle


\begin{abstract}
For $0 \leq \alpha < n$ and $m \in \mathbb{N} \cap (1 - \frac{\alpha}{n}, \, \infty)$, we introduce a class of fractional 
series operators $T_{\alpha, m}$ defined on $\mathbb{Z}^n$ which are generated by certain $m$-invertible matrices with integer 
coefficients. In this note, we prove that $T_{\alpha, m}$ is a bounded operator $H^p(\mathbb{Z}^n) \to \ell^q(\mathbb{Z}^n)$ for 
$0 < p < \frac{n}{\alpha}$ and $\frac{1}{q} = \frac{1}{p} - \frac{\alpha}{n}$. This generalizes the results obtained by the author in 
[Acta Math. Hungar., 168 (1) (2022), 202-216].
\end{abstract}

\section{Introduction}

Given $0 \leq \alpha < n$, let $m \in \mathbb{N} \cap (1 - \frac{\alpha}{n}, \, \infty)$ and let $\alpha_1, ..., \alpha_m$ be $m$ positive constants such that $\alpha_1 + \cdot \cdot \cdot + \alpha_m = n - \alpha$. We define the discrete operator 
$T_{\alpha, m}$ on $\mathbb{Z}^n$ by
\begin{equation} \label{Tam}
(T_{\alpha, m} b)(j) = \sum_{i \neq A_k j \, : \, k=1, ..., m} \frac{b(i)}{|i - A_1 j|^{\alpha_1} \cdot \cdot \cdot |i - A_m j|^{\alpha_m}},
\,\,\,\,\,\, j \in \mathbb{Z}^n,
\end{equation}
where the $A_k$'s are invertible matrices of $M_n(\mathbb{Z})$ ($=$ the set of all matrices of degree $n$ with coefficients in 
$\mathbb{Z}$). For the case $\alpha = 0$, we also assume that $A_{k} - A_{l}$ is invertible if $1 \leq k \neq l \leq m$.

The case when $n=1$, $0 \leq \alpha <1$, $m=2$, $A_1 = 1$ and $A_2 = -1$ in (\ref{Tam}) was studied by the author in \cite{Rocha1}. 
We proved, using the atomic decomposition for $H^p(\mathbb{Z})$ given in \cite{Boza}, that such operator is bounded from $H^p(\mathbb{Z})$ into $\ell^q(\mathbb{Z})$ for $0 < p < \frac{1}{\alpha}$ and $\frac{1}{q} = \frac{1}{p} - \alpha$ (see Theorems 7 and 9 in \cite{Rocha1}).

We also observe that the operator (\ref{Tam}) is a generalization of the discrete Riesz potential on $\mathbb{Z}^n$. Indeed, for 
$0 < \alpha < n$, $m=1$ and $A_1 = Id$, we have that $T_{\alpha, 1} = I_{\alpha}$, where
\begin{equation} \label{Riesz potential}
(I_{\alpha}b)(j) = \sum_{i \in \mathbb{Z}^n \setminus \{ j \}} \frac{b(i)}{|i-j |^{n - \alpha}}, \,\,\,\,\,\, j \in \mathbb{Z}^n.
\end{equation}
Y. Kanjin and M. Satake in \cite{Kanjin} studied the discrete Riesz potential $I_{\alpha}$ for the case $n=1$ and proved the 
$H^p(\mathbb{Z}) \to H^q(\mathbb{Z})$ boundedness of $I_{\alpha}$, for $0 < p < \frac{1}{\alpha}$ and 
$\frac{1}{q} = \frac{1}{p} - \alpha$. To achieve this result, they furnished a molecular decomposition for elements of $H^p(\mathbb{Z})$.

Recently, by means of the atomic decomposition for $H^p(\mathbb{Z}^n)$ given in \cite{Carro}, the author in \cite{Rocha2} and \cite{Rocha3} studied the behavior of discrete Riesz potential on $H^p(\mathbb{Z}^n)$. More precisely, in \cite{Rocha2} we proved the $H^p(\mathbb{Z}^n) \to \ell^q(\mathbb{Z}^n)$ boundedness of $I_{\alpha}$ for 
$0 < p < \frac{n}{\alpha}$ and $\frac{1}{q} = \frac{1}{p} - \frac{\alpha}{n}$; in \cite{Rocha3}, on the range $\frac{n-1}{n} < p \leq 1$, 
we furnished a molecular decomposition for $H^p(\mathbb{Z}^n)$ analogous to the ones given by Y. Kanjin and M. Satake in \cite{Kanjin}, and obtained the $H^p(\mathbb{Z}^n) \to H^q(\mathbb{Z}^n)$ boundedness of $I_{\alpha}$ for $\frac{n-1}{n} < p < q \leq 1$.

In \cite{Rocha1}, we also showed that there exists $\epsilon \in (0, \frac{1}{3})$ such that, for every $0 \leq \alpha < \epsilon$, the 
operator $T_{\alpha, m}$ given by (\ref{Tam}), with $n=1$, $m=2$, $\alpha_1 = \alpha_2 = \frac{1-\alpha}{2}$, $A_1 = 1$ and $A_2 = -1$, is not bounded from $H^p(\mathbb{Z})$ into $H^q(\mathbb{Z})$ for $0 < p \leq \frac{1}{1+\alpha}$ and $\frac{1}{q} = \frac{1}{p} - \alpha$. This is a significant difference with respect to the case $0 < \alpha < 1$, $n = m =1$ and $A_1 = 1$ (i.e: $T_{\alpha, 1} = I_{\alpha}$ is discrete Riesz potential on $\mathbb{Z}$).

For more results about discrete fractional type operators one can consult \cite{Hardy}, \cite{Wainger}, and \cite{Oberlin}. On the other hand, the $H^p(\mathbb{R}^n) \to L^q(\mathbb{R}^n)$ boundedness of the continuous counterpart of (\ref{Tam}) was studied by the author and M. Urciuolo in \cite{RU1} and \cite{RU2} (see also \cite{RU3} and \cite{Rocha}),

\vspace{0.5cm}

The main aim of this note is to prove the following two results. These generalize the Theorems 7 and 9 in \cite{Rocha1} respectively.

\vspace{0.5cm}

{\sc Theorem} \ref{lplq}. {\it For $0 \leq \alpha < n$ and $m \in  \mathbb{N} \cap (1 - \frac{\alpha}{n}, \infty)$, let $T_{\alpha, m}$ be the discrete operator given by (\ref{Tam}). If $1 < p < \frac{n}{\alpha}$ and $\frac{1}{q} = \frac{1}{p} - \frac{\alpha}{n}$, then there exists a positive constant $C$ such that
\[
\| T_{\alpha, m} b \|_{\ell^q(\mathbb{Z}^n)} \leq C \| b \|_{\ell^p(\mathbb{Z}^n)},
\]
for all $b \in \ell^p(\mathbb{Z}^n)$.}

\

{\sc Theorem} \ref{Hplq}. {\it For $0 \leq \alpha < n$ and $m \in  \mathbb{N} \cap (1 - \frac{\alpha}{n}, \infty)$, let $T_{\alpha, m}$ be the operator given by (\ref{Tam}). Then, for $0 < p \leq 1$ and $\frac{1}{q} = \frac{1}{p} - \frac{\alpha}{n}$
\[
\| T_{\alpha, m} \, b \|_{\ell^{q}(\mathbb{Z}^n)} \leq C \| b \|_{H^{p}(\mathbb{Z}^n)},
\]
where $C$ does not depend on $b$.}

\

{\bf Notation.} We set $\mathbb{N}_0 = \mathbb{N} \cup \{0\}$. For every $E \subset \mathbb{Z}^n$, we denote by $\#E$ and $\chi_{E}$ 
the cardinality of the set $E$ and the characteristic sequence of $E$ on $\mathbb{Z}^n$ respectively. Given a real number $s \geq 0$, 
we write $\lfloor s \rfloor$ for the integer part of $s$. As usual we denote with $\mathcal{S}(\mathbb{R}^{n})$ the space of smooth 
and rapidly decreasing functions. If $i = (i_1, ..., i_n) \in \mathbb{Z}^n$ and $\beta$ is the multi-index 
$\beta =(\beta_{1},...,\beta _{n})$, then $i^{\beta} := i_1^{\beta_1} \cdot \cdot \cdot i_n^{\beta_n}$ and
$[\beta]:=\beta _{1} + \cdot \cdot \cdot + \beta _{n}$. Throughout this paper, $C$ will denote a positive real constant not necessarily 
the same at each occurrence.

\section{Preliminaries}

This section presents three auxiliary results necessary to obtain the main results of Sections 3 and 4. 

For $0 < p < \infty$ and a sequence $b = \{ b(i) \}_{i \in \mathbb{Z}^n}$ we say that $b$ belongs to $\ell^{p}(\mathbb{Z}^n)$ if
$$\| b \|_{\ell^p(\mathbb{Z}^n)} :=\left( \sum_{i \in \mathbb{Z}^n} |b(i)|^p \right)^{1/p} < \infty.$$ For $p=\infty$, we say that $b$ belongs to 
$\ell^{\infty}(\mathbb{Z}^n)$ if $$\|b \|_{\ell^\infty(\mathbb{Z}^n)} := \sup_{i \in \mathbb{Z}^n} |b(i)| < \infty.$$

In the sequel, for $j =(j_1, ..., j_n) \in \mathbb{Z}^{n}$ we put $|j|_{\infty} = \max \{ |j_k| : k=1, ..., n \}$ and 
$|j| = (j_1^2 + \cdot \cdot \cdot + j_n^2 )^{1/2}$.

A discrete cube $Q$ centered at $i_0 \in \mathbb{Z}^n$ is of the form $Q = \{ i \in \mathbb{Z}^n : |i - i_0|_{\infty} \leq N \}$, where 
$N \in \mathbb{N}_0$. It is clear that $\# Q = (2N+1)^n$. 

Let $0 \leq \alpha < n$, given a sequence $b = \{ b(i) \}_{i \in \mathbb{Z}^n}$ we define the centered fractional maximal sequence 
$M_{\alpha} b$ by
\[
(M_{\alpha} b)(j) = \sup_{Q \ni j} \frac{1}{\# Q^{1 - \frac{\alpha}{n}}} \sum_{i \in Q} |b(i)|, \,\,\,\,\,\, j \in \mathbb{Z}^n,
\]
where the supremum is taken over all discrete cubes $Q$ centered at $j$. We observe that if $\alpha = 0$, then $M_0 = M$ where $M$ is the centered discrete maximal operator.

The following result was proved in \cite[see Theorem 2.3 and Proposition 2.4]{Rocha2}.

\begin{proposition} \label{fract max}
Let $0 \leq \alpha < n$. If $1 < p < \frac{n}{\alpha}$ and $\frac{1}{q} = \frac{1}{p} - \frac{\alpha}{n}$, then
\[
\| M_{\alpha} b \|_{\ell^q(\mathbb{Z}^n)} \leq C \| b \|_{\ell^p(\mathbb{Z}^n)}, \,\,\,\, \forall \,\, b \in \ell^p(\mathbb{Z}^n).
\]
\end{proposition}

The following lemma is crucial to get the Theorem \ref{lplq}.

\begin{lemma} (\cite[Lemma 2.1]{Rocha3})\label{series0} If $\epsilon > 0$ and $N \in \mathbb{N}$, then 
\begin{equation}
\sum_{|j|_{\infty} \geq N} \frac{1}{|j|^{n+\epsilon}} \leq 2^{n}n^{n+\epsilon} \left(2+\frac{2^{\frac{\epsilon}{n}} n}{\epsilon} \right)^{n} 
N^{-\epsilon}.
\end{equation}
\end{lemma}

Next, we introduce the definition of $p$-atom.

\begin{definition} \label{p atom} Let $0 < p \leq 1$ and $d_p := \lfloor n(p^{-1} - 1) \rfloor$. We say that a sequence 
$a = \{ a(i) \}_{i \in \mathbb{Z}^n}$ is an $(p, \infty, d_p)$-atom centered at a discrete cube $Q \subset \mathbb{Z}^n$ if following three conditions hold:

(a1) $\supp a \subset Q$,

(a2) $\| a \|_{\ell^\infty(\mathbb{Z}^n)} \leq (\# Q)^{-1/p}$,

(a3) $\displaystyle{\sum_{i \in Q}} i^{\beta} a(i) = 0$ for every multi-index $\beta=(\beta_1, ..., \beta_n) \in \mathbb{N}_0^n$ with 
$[\beta] \leq d_p$.
\end{definition}

Given a discrete cube $Q = \{ i \in \mathbb{Z}^n : |i - i_0|_{\infty} \leq N \}$ and $m$-invertible matrices $A_1, ..., A_m$ belonging to 
$M_n(\mathbb{Z})$, we define, for $k=1, ..., m$, $Q^{*}_k = \{ i \in \mathbb{Z}^n : |i - A^{-1}_k i_0|_{\infty} \leq 4 D N \}$, where
$D = \max \{ \|A_k^{-1} \| : k=1, ..., m  \}$. Then, we put $R = \mathbb{Z}^n \setminus \left( \bigcup_{k=1}^{m} Q^{*}_k \right) = \left( \bigcup_{k=1}^{m} Q^{*}_k \right)^c$. Moreover $R = \bigcup_{l=1}^m R_l$, where
\[
R_l = \{ j \in R : |j - A_l^{-1} i_0| \leq |j - A_k^{-1} i_0| \,\,\, \text{for all} \,\, l \neq k \},
\]
for every $l = 1, ..., m$.

By adapting the argument used in the proof of Lemma 14 in \cite{Rocha} to our setting, we obtain the following result.

\begin{lemma} \label{estim on Rl}
For $0 \leq \alpha < n$ and $m \in  \mathbb{N} \cap (1 - \frac{\alpha}{n}, \infty)$, let $T_{\alpha, m}$ be the discrete operator given 
by (\ref{Tam}). If $a = \{ a(i) \}_{i \in \mathbb{Z}^n}$ is an $(p, \infty,d_p)$-atom centered at a discrete cube $Q \subset \mathbb{Z}^n$, then
\[
|T_{\alpha, m}a(j)| \leq C \|a \|_{\ell^{\infty}} \sum_{l=1}^m \chi_{R_l}(j)
\left( M_{\frac{\alpha n}{n+d_p+1}}(\chi_Q)(A_l \, j) \right)^{\frac{n+d_p+1}{n}}, \,\,\,\, \text{if} \,\, j \in R,
\]
where $C$ does not depend on $a$.
\end{lemma}

\section{The $\ell^p(\mathbb{Z}^n) - \ell^q(\mathbb{Z}^n)$ boundedness of $T_{\alpha, m}$}

In this section we establish the $\ell^p(\mathbb{Z}^n) - \ell^q(\mathbb{Z}^n)$ boundedness of the discrete operator $T_{\alpha, m}$, for 
$1 < p < \frac{n}{\alpha}$ and $\frac{1}{q} = \frac{1}{p} - \frac{\alpha}{n}$.

\begin{theorem} \label{lplq}
For $0 \leq \alpha < n$ and $m \in  \mathbb{N} \cap (1 - \frac{\alpha}{n}, \infty)$, let $T_{\alpha, m}$ be the discrete operator given 
by (\ref{Tam}). If $1 < p < \frac{n}{\alpha}$ and $\frac{1}{q} = \frac{1}{p} - \frac{\alpha}{n}$, then there exists a positive constant $C$ such that
\[
\| T_{\alpha, m} b \|_{\ell^q(\mathbb{Z}^n)} \leq C \| b \|_{\ell^p(\mathbb{Z}^n)},
\]
for all $b \in \ell^p(\mathbb{Z}^n)$.
\end{theorem}

\begin{proof} Given a sequence $b = \{ b(i) \}_{i \in \mathbb{Z}^n}$ we put $|b| = \{ |b(i)| \}_{i \in \mathbb{Z}^n}$. We study the cases 
$0 < \alpha < n$ and $\alpha = 0$ separately. For $0 < \alpha < n$ it is easy to check that
\[
|(T_{\alpha, m} \, b)(j)| \leq \sum_{k=1}^{m} (I_{\alpha}|b|)(A_k j), \,\,\,\, \forall j \in \mathbb{Z}^n.
\]
We observe that $\sum_{j \in \mathbb{Z}^n} |(I_{\alpha}|b|)(A j)|^q \leq  \sum_{j \in \mathbb{Z}^n} |(I_{\alpha}|b|)(j)|^q$ for 
every invertible matrix $A \in M_n(\mathbb{Z})$, since $A(\mathbb{Z}^n) \subset \mathbb{Z}^n$. Thus, the 
$\ell^{p}(\mathbb{Z}^n)$ - $\ell^{q}(\mathbb{Z}^n)$ boundedness of $T_{\alpha, m}$ ($0 < \alpha < n$) follows from Theorem 
\cite[Proposition (a)]{Wainger} or \cite[Theorem 3.1]{Rocha2}.

For $\alpha = 0$, we have that $m \geq 2$, without loss of generality, we may consider $m=2$. We point out that this case is entirely 
representative for the general case $m \geq 2$. Now, we introduce the auxiliary operator $\widetilde{T}$ defined by 
$(\widetilde{T} b)(j) = (T_{0, 2} b)(j)$ if $j \neq {\bf 0}$ and $(\widetilde{T} b)({\bf 0}) = 0$. From H\"older inequality and 
Lemma \ref{series0} applied with $N=1$ and $\epsilon = n(p' - 1)$, we have that
\[
|(T_{0, 2} b)({\bf 0})| \leq \| \{ |i|^{-n} \} \|_{\ell^{p'}(\mathbb{Z}^n \setminus \{ {\bf 0 }\})} \| b \|_{\ell^{p}(\mathbb{Z}^n)}
< \infty, \,\,\,\,\,\, \text{for all} \,\, 1 \leq p < \infty.
\]
So, it suffices to show that $\widetilde{T}$ is bounded on $\ell^{p}(\mathbb{Z}^n)$, $1 < p < +\infty$. For them, we put
$d = \min \{ |A_1 x - A_2 x| : |x|=1 \}$ and $D = \max \{ \| A_1 \|, \| A_2 \|  \}$. Since the matrices $A_1$, $A_2$ and $A_1 - A_2$ are invertible with integer coefficients we have that $d > 0$ and $D \geq 1$. For $j_0 \in \mathbb{Z}^n \setminus \{{\bf 0} \}$,  we write 
$\mathbb{Z}^n \setminus \{ A_1 j_0, A_2 j_0 \} = I_1 \cup I_2 \cup I_3 \cup I_4$, where
\[
I_k = \left\{ i \in \mathbb{Z}^n : 0 < |i - A_k j_0| \leq \frac{d}{2} |j_0| \right\}, \,\,\,\,\, \text{for} \, k=1,2,
\] 
\[
I_3 = \{ i \in \mathbb{Z}^n : |i| < 2 \sqrt{n} D |j_0| \} \cap (I_1^c \cap I_2^c), \,\,\,\, \text{and} \,\,\,\, 
I_4 = \{ i \in \mathbb{Z}^n : |i| \geq 2 \sqrt{n} D |j_0| \} \cap (I_1^c \cap I_2^c).
\]
Then,
\[
|(\widetilde{T} b)(j_0)|=|(T_{\alpha, 2} \, b)(j_0)| \leq \left( \sum_{i \in I_1} + \sum_{i \in I_2} + 
\sum_{i \in I_3} + \sum_{i \in I_4} \right) \frac{|b(i)|}{|i- A_1j_0|^{\alpha_1} |i- A_2j_0|^{\alpha_2}}.
\]
First, we estimate the sum on $I_1$. If $i \in I_1$, then $|i - A_2 j_0| = |A_1 j_0 - A_2 j_0 + i - A_1 j_0| \geq \frac{d}{2} |j_0|$. So,
\[
\sum_{i \in I_1} \frac{|b(i)|}{|i- A_1 j_0|^{\alpha_1} |i - A_2 j_0|^{\alpha_2}} \leq \frac{2^{\alpha_2} }{d^{\alpha_2}|j_0|^{\alpha_2}} 
\sum_{0 < |i- A_1 j_0| \leq \frac{d}{2}|j_0|} \frac{|b(i)|}{|i-A_1 j_0|^{\alpha_1}} =:{\sum}_1.
\]
Now, we take $k_0 \in \mathbb{N}_0$ such that $2^{k_0} \leq \frac{d}{2}|j_0| < 2^{k_0 + 1}$, thus
\[
{\sum}_1 = \sum_{k=0}^{k_0} \frac{2^{\alpha_2}}{d^{\alpha_2}|j_0|^{\alpha_2}} \sum_{2^{-(k+2)}d|j_0| < |i- A_1 j_0| \leq 2^{-(k+1)}d|j_0|} 
\frac{|b(i)|}{|i- A_1 j_0|^{\alpha_1}}
\] 
\[
\leq 2^{\alpha_2}\sum_{k=0}^{k_0} \frac{2^{(k+2)\alpha_1}}{d^{n}|j_0|^n} \sum_{|i- A_1 j_0| \leq \lfloor 2^{-(k+1)} d |j_0| \rfloor} |b(i)|
\]
\[
= 2^{\alpha_2 + 2\alpha_1} \sum_{k=0}^{k_0} \frac{2^{-\alpha_2 k}}{(2 \cdot 2^{-(k+1)} d |j_0|)^n} 
\sum_{|i- A_1 j_0| \leq \lfloor 2^{-(k+1)}d|j_0| \rfloor} |b(i)|
\]
\[
\leq 2^{\alpha_2 + 2\alpha_1} \sum_{k=0}^{k_0} 2^{-\alpha_2 k} \frac{1}{(2 \cdot \lfloor 2^{-(k+1)} d |j_0| \rfloor + 1)^n}  
\sum_{|i- A_1 j_0| \leq \lfloor 2^{-(k+1)}d|j_0| \rfloor} |b(i)|,
\]
this last inequality follows from that $\lfloor 2^{-(k+1)} d |j_0| \rfloor \leq 2^{-(k+1)} d |j_0|$ and that
$\frac{2 \cdot \lfloor 2^{-(k+1)} d |j_0| \rfloor + 1}{2 \cdot \lfloor 2^{-(k+1)} d |j_0| \rfloor} \leq 2$ for each $k=0, ..., k_0$. Thus
\begin{equation} \label{I1}
\sum_{i \in I_1} \frac{|b(i)|}{|i- A_1 j_0|^{\alpha_1} |i - A_2 j_0|^{\alpha_2}} \leq 2^{\alpha_2 + 2\alpha_1} 
\left( \sum_{k=0}^{\infty} 2^{-\alpha_2 k} \right) (Mb)(A_1 j_0).
\end{equation}
Similarly, it is seen that
\begin{equation} \label{I2}
\sum_{i \in I_2} \frac{|b(i)|}{|i- A_1 j_0|^{\alpha_1} |i- A_2 j_0|^{\alpha_2}} \leq C (Mb)(A_2 j_0).
\end{equation}
On $I_3$ we obtain,
\begin{equation} \label{I3}
\sum_{i \in I_3} \frac{|b(i)|}{|i- A_1 j_0|^{\alpha_1} |i- A_2 j_0|^{\alpha_2}} \leq 
\frac{2^n}{d^n} |j_0|^{-n} \sum_{|i| < 2 \sqrt{n} D |j_0|} |b(i)|
\end{equation}
\[
\leq \frac{2^n}{d^n} |j_0|^{-n} \sum_{|i-j_0| \leq (2 \sqrt{n} D+1) |j_0|} |b(i)| \leq C (Mb)(j_0).
\]
Now, on $I_4$ we have, for every $k=1,2$, that $|i - A_k j_0| \geq \frac{(2 \sqrt{n} D - 1)}{2 \sqrt{n} D} |i|$ for all $i \in I_4$. Since 
$I_4 \subset \{ i \in \mathbb{Z}^n : |i| \geq 2 \sqrt{n} D |j_0| \} \subset \{ i \in \mathbb{Z}^n : |i|_{\infty} \geq 2 \lfloor D \rfloor 
|j_0| \}$, it follows that
\begin{equation} \label{estim}
\sum_{i \in I_4} \frac{|b(i)|}{|i- A_1 j_0|^{\alpha_1} |i- A_2 j_0|^{\alpha_2}} \leq C \sum_{|i|_{\infty} \geq 2 \lfloor D \rfloor |j_0|} 
|i|^{-n} |b(i)| \leq C \| b \|_{\ell^{p}} |j_0|^{-n/p} \leq C \| b \|_{\ell^{p}} |j_0|^{-n/p}_{\infty},
\end{equation}
where the second inequality follows from H\"older's inequality and Lemma \ref{series0} applied with $N= 2 \lfloor D \rfloor |j_0|$ and 
$\epsilon = n (p' - 1)$. 
Thus (\ref{estim}) implies that
\begin{equation} \label{I4}
\# \left\{ j \neq 0: \left| \sum_{i \in I_4} |i- A_1 j|^{-\alpha_1}|i - A_2 j|^{-\alpha_2} b(i) \right| > \lambda \right\} \leq 
\left(C \frac{\|b\|_{\ell^{p}}}{\lambda} \right)^{p}, \,\, 1 \leq p < \infty.
\end{equation}
Finally, (\ref{I1}), (\ref{I2}), (\ref{I3}), (\ref{I4}) and Proposition \ref{fract max} with $\alpha=0$ allow us to conclude that 
$\widetilde{T}$ is a bounded operator $\ell^{p}(\mathbb{Z}^n) \to \ell^{p, \infty}(\mathbb{Z}^n)$, for every 
$1 \leq p < \infty$. Then, the $\ell^{p}(\mathbb{Z})$ boundedness of $\widetilde{T}$ follows from the Marcinkiewicz interpolation theorem 
(see Theorem 1.3.2 in \cite{Grafakos}). This completes the proof.
\end{proof}

\begin{remark}
Let $0 \leq \alpha < n$. If $\frac{n}{n-\alpha} < q < \infty$ and $0 < p \leq \frac{nq}{n + \alpha q}$, then the operator 
$T_{\alpha, m}$ is bounded from $\ell^{p}(\mathbb{Z}^n)$ into $\ell^{q}(\mathbb{Z}^n)$. This follows from Theorem \ref{lplq} and 
the embedding $\ell^{p_1}(\mathbb{Z}^n) \hookrightarrow \ell^{p_2}(\mathbb{Z}^n)$ valid for $0 < p_1 < p_2 \leq \infty$.

\end{remark}

\section{The $H^p(\mathbb{Z}^n) - \ell^q(\mathbb{Z}^n)$ boundedness of $T_{\alpha, m}$}

Firstly, we recall the definition of $H^p(\mathbb{Z}^n)$ spaces and state the atomic decomposition given by S. Boza and M. Carro 
in \cite{Carro}. 

Let $\Phi \in \mathcal{S}(\mathbb{R}^n)$ with $\int_{\mathbb{R}^n} \Phi =1$, $\Phi^d$ denotes the restriction of $\Phi$ on $\mathbb{Z}^n$. Now, for $t>0$, we consider $\Phi_t^d(j) = t^{-n} \Phi(j/t)$ if $j \neq {\bf 0}$ and $\Phi_t^d({\bf 0}) = 0$. Then, by 
\cite[Theorem 2.7]{Carro}, we define
\[
H^p(\mathbb{Z}^n) = \left\{ b \in \ell^p(\mathbb{Z}^n) : \sup_{t>0} |(\Phi_t^d \ast_{\mathbb{Z}^n} b)| \in \ell^p(\mathbb{Z}^n) \right\},
\,\,\, 0 < p \leq 1,
\]
with the "$H^p(\mathbb{Z}^n)$-norm" given by
\[
\| b \|_{H^p(\mathbb{Z}^n)} := \| b \|_{\ell^p(\mathbb{Z}^n)} + \|\sup_{t>0} |(\Phi_t^d \ast_{\mathbb{Z}^n} b)| \|_{\ell^p(\mathbb{Z}^n)}.
\]

From Definition \ref{p atom}, we have that if $a = \{ a(j) \}_{j \in \mathbb{Z}^n}$ is an $(p, \infty, d_p)$-atom, then 
$a = \{ a(j) \}_{j \in \mathbb{Z}^n} \in H^p(\mathbb{Z}^n)$. The atomic decomposition for $H^p(\mathbb{Z}^n)$, $0 < p \leq 1$, developed 
in \cite{Carro} is as follows:

\begin{theorem} (\cite[Theorem 3.7]{Carro}) \label{atomic Hp} Let $0 < p \leq 1$, $d_p = \lfloor n (p^{-1} - 1) \rfloor$ and 
$b \in H^{p}(\mathbb{Z}^n)$. Then there exist a sequence of $(p, \infty, d_p)$-atoms $\{ a_k \}_{k=0}^{+\infty}$, a sequence of scalars 
$\{ \lambda_k \}_{k=0}^{+\infty}$ and a positive constant $C$, which depends only on $p$ and $n$, with 
$\sum_{k=0}^{+\infty} |\lambda_k |^{p} \leq C \| b \|_{H^{p}(\mathbb{Z}^n)}^{p}$ such that $b = \sum_{k=0}^{+\infty} \lambda_k a_k$, where the series converges in $H^{p}(\mathbb{Z}^n)$.
\end{theorem}

Now, we are in a position to prove our main result.

\begin{theorem} \label{Hplq} For $0 \leq \alpha < n$ and $m \in  \mathbb{N} \cap (1 - \frac{\alpha}{n}, \infty)$, let $T_{\alpha, m}$ be the operator given by (\ref{Tam}). Then, for $0 < p \leq 1$ and $\frac{1}{q} = \frac{1}{p} - \frac{\alpha}{n}$
\[
\| T_{\alpha, m} \, b \|_{\ell^{q}(\mathbb{Z}^n)} \leq C \| b \|_{H^{p}(\mathbb{Z}^n)},
\]
where $C$ does not depend on $b$.
\end{theorem}

\begin{proof}
For $0 < p \leq 1$ and $\frac{1}{q} = \frac{1}{p} - \frac{\alpha}{n}$, we shall prove that there exists an universal positive constant $C$ such that
\begin{equation} \label{uniform estim}
\| T_{\alpha, m} a \|_{\ell^q} \leq C,
\end{equation}
for all $(p, \infty, d_p)$-atom $a= \{ a(i) \}_{i \in \mathbb{Z}^n}$. For them, we consider an $(p, \infty, d_p)$-atom $a= \{ a(i) \}_{i \in \mathbb{Z}^n}$ supported on the discrete cube $Q= \{ i \in \mathbb{Z}^n : |i - i_0|_{\infty} \leq N \}$. For every $k=1, ..., m$, let 
$Q^{*}_k = \{ i \in \mathbb{Z}^n : |i - A^{-1}_k i_0|_{\infty} \leq 4 D N \}$, where $D = \max \{ \|A_k^{-1} \| : k=1, ..., m  \}$. Now, we decompose $\mathbb{Z}^n = \left( \bigcup_{k=1}^m Q^{*}_k \right) \cup R$, where $R = \left( \bigcup_{k=1}^m Q^{*}_k \right)^c$. So,
\[
\sum_{j \in \mathbb{Z}^n} |(T_{\alpha, m} a)(j)|^q \leq \sum_{k=1}^m \sum_{j \in Q^{*}_k} |(T_{\alpha, m} a)(j)|^q +
\sum_{j \in R} |(T_{\alpha, m} a)(j)|^q = I_1 + I_2.
\]
To estimate $I_1$ we take $\frac{n}{n-\alpha} < q_0 < \infty$ and put $\frac{1}{p_0} = \frac{1}{q_0} + \frac{\alpha}{n}$. By H\"older inequality applied with $q_0/q$ and Theorem \ref{lplq}, we obtain
\begin{equation} \label{estim I1}
I_1 \leq \| T_{\alpha, m} a \|_{\ell^{q_0}}^q \sum_{k=1}^m (\# Q^{*}_k)^{1-q/q_0} \leq C \| a \|_{\ell^{p_0}}^q 
\sum_{k=1}^m (\# Q^{*}_k)^{1-q/q_0}
\end{equation}
\[
\leq C (\# Q)^{-q/p} (\# Q)^{q/p_0} \sum_{k=1}^m (\# Q^{*}_k)^{1-q/q_0} \leq C,
\]
where $C$ does not depend on the $p$-atom $a$.

Now, we proceed to estimate $I_2$. By Lemma \ref{estim on Rl}, we have
\begin{equation} \label{estimate I2}
I_2 \leq C \|a \|_{\ell^{\infty}}^q \sum_{l=1}^m \sum_{j \in \mathbb{Z}^n} 
\left( M_{\frac{\alpha n}{n+d_p+1}}(\chi_Q)(A_l \, j) \right)^{q\frac{n+d_p+1}{n}}.
\end{equation}
Since $A_l(\mathbb{Z}^n) \subset \mathbb{Z}^n$ for every $l=1, ..., m$ and the $A_l$'s are invertible, it follows that
\begin{equation} \label{reo}
\sum_{j \in \mathbb{Z}^n} 
\left( M_{\frac{\alpha n}{n+d_p+1}}(\chi_Q)(A_l \, j) \right)^{q\frac{n+d_p+1}{n}} \leq \sum_{j \in \mathbb{Z}^n} 
\left( M_{\frac{\alpha n}{n+d_p+1}}(\chi_Q)(j) \right)^{q\frac{n+d_p+1}{n}}.
\end{equation}
By taking into account that $d_p= \lfloor n(\frac{1}{p}-1) \rfloor$, we have $q \frac{n+d+1}{n} > p \frac{n+d+1}{n} > 1$. Then, we write 
$\widetilde{q} = q \frac{n+d+1}{n}$ and let $\frac{1}{\widetilde{p}} = \frac{1}{\widetilde{q}} + \frac{\alpha}{n+d+1}$, so 
$1 < \widetilde{p} < \widetilde{q} < \infty$ and $\widetilde{p}/\widetilde{q} = p/q$. Then, Proposition \ref{fract max} leads to
\[
\sum_{j \in \mathbb{Z}^n} \left( M_{\frac{\alpha n}{n+d_p+1}}(\chi_Q)(j) \right)^{q\frac{n+d_p+1}{n}} 
\leq C \left( \sum_{j \in \mathbb{Z}^n} \chi_Q (j) \right)^{q/p} = C (\# Q)^{q/p}.
\]
This inequality, (\ref{reo}) and (\ref{estimate I2}) give
\begin{equation} \label{estim I2}
I_2 \leq C \| a \|_{\ell^{\infty}}^q (\# Q)^{q/p} = C, 
\end{equation}
where $C$ is independent of the $p$-atom $a$. Now, (\ref{estim I1}) and (\ref{estim I2}) allow us to obtain (\ref{uniform estim}).

Given $b \in H^p(\mathbb{Z}^n)$, by Theorem \ref{atomic Hp}, we can write $b = \sum \lambda_k a_k$ where the $a_k$'s are discrete
$(p, \infty, d_p)$ atoms and the scalars $\lambda_k$ satisfies $\sum_{k} |\lambda_k |^{p} \leq C \| b \|_{H^{p}(\mathbb{Z}^{n})}$. By Theorem \ref{lplq} applied with $\frac{n}{n - \alpha} < q_0 < \infty$ and $\frac{1}{p_0} = \frac{1}{q_0} + \frac{\alpha}{n}$ and since 
$b = \sum_k \lambda_k a_k$ converges in $\ell^{p_0}(\mathbb{Z}^{n})$, we have that
\begin{equation} \label{pointwise}
|(T_{\alpha, m}b)(j)| \leq \sum_{k=1}^{\infty} |\lambda_k| |(T_{\alpha, m}a_k)(j)|, \,\,\,\,\, \text{for all} \,\, j \in \mathbb{Z}^{n}. 
\end{equation}
Finally, (\ref{uniform estim}) and (\ref{pointwise}) allows us to obtain
\[
\|T_{\alpha, m} b \|_{\ell^{q}(\mathbb{Z}^{n})} \leq C \left( \sum_{k} |\lambda_k|^{\min\{1, q \}} \right)^{\frac{1}{\min\{1, q \}}} \leq C \left( \sum_{k} |\lambda_k |^{p} \right)^{1/p} \leq C \| b \|_{H^{p}(\mathbb{Z}^{n})}.
\]
Thus the proof is concluded.
\end{proof}

In the following corollary we recover Theorem 3.3 obtained in \cite{Rocha2}.

\begin{corollary}
For $0 < \alpha < n$, let $I_{\alpha}$ be the discrete Riesz potential given by (\ref{Riesz potential}). Then, for
$0 < p \leq 1$ and $\frac{1}{q} = \frac{1}{p} - \frac{\alpha}{n}$
\[
\| I_{\alpha} \, b \|_{\ell^{q}(\mathbb{Z}^n)} \leq C \| b \|_{H^{p}(\mathbb{Z}^n)},
\]
where $C$ does not depend on $b$.
\end{corollary}

\begin{proof}
To apply Theorem \ref{Hplq} with $0 < \alpha < n$, $m=1$ and $A_1 = Id$.
\end{proof}

\begin{remark}
Let $0 \leq \alpha < n$. If $0 < q \leq \frac{n}{n-\alpha}$ and $0 < p \leq \frac{nq}{n + \alpha q}$, then the operator $T_{\alpha, m}$ is bounded from $H^{p}(\mathbb{Z}^n)$ into $\ell^{q}(\mathbb{Z}^n)$. This follows from Theorem \ref{Hplq} and the embedding 
$H^{p_1}(\mathbb{Z}^n) \hookrightarrow H^{p_2}(\mathbb{Z}^n)$ valid for $0 < p_1 < p_2 \leq 1$. In particular, for $0 < \alpha < n$,
$0 < q \leq \frac{n}{n-\gamma}$ and $0 < p \leq \frac{nq}{n + \alpha q}$, the discrete Riesz potential $I_{\alpha}$ is bounded 
from $H^{p}(\mathbb{Z}^n)$ into $\ell^{q}(\mathbb{Z}^n)$.
\end{remark}

\address{
Departamento de Matem\'atica \\ 
Universidad Nacional del Sur (UNS) \\
Bah\'{\i}a Blanca, Argentina}
{pablo.rocha@uns.edu.ar}

\end{document}